\documentclass[12pt]{article}
\usepackage{amsmath,amssymb,setspace,amsthm,amsfonts,diagrams}

\def\Der{\mathop{\rm Der}} \def\sl{\mathop{\mathfrak sl}}
\def\Span{\mathop{\rm span}} \def\Hom{{\rm Hom}} \def\ind{\mathop{\rm
    ind}} \def\ad{\mathop{\rm ad}} \def\im{\mathop{\rm im}}
\def\ind{{\rm ind}} \def\K{\mathbb{K}} \def\phi{\varphi}
\makeatletter
\let\@@pmod\pmod
\DeclareRobustCommand{\pmod}{\@ifstar\@pmods\@@pmod}
\def\@pmods#1{\mkern4mu({\operator@font mod}\mkern 6mu#1)}
\makeatother

\newtheorem{theorem}{Theorem}[section]
\newtheorem{lemma}[theorem]{Lemma}

 \theoremstyle{remark}
\newtheorem*{remark}{Remark}

\title{Restricted Cohomology of Modular Witt Algebras}

\author{Tyler J. Evans \\Department of Mathematics \\Humboldt State
  University \\Arcata, CA 95521 USA \\evans@humboldt.edu \and Alice
  Fialowski \\Institute of Mathematics \\E\" otv\" os Lor\' and
  University \\H-1117 Budapest Hungary\\ fialowsk@cs.elte.hu \and
  Michael Penkava \\Department of Mathematics \\University of
  Wisconsin-Eau Claire \\Eau Claire, WI 54701 USA \\penkavmr@uwec.edu}

\begin{document}

\maketitle

\section{Introduction}
Let $p>3$ be a prime, $\K$ be an algebraically closed field of
characteristic $p$ and $A=\K[x]/(x^p-1)$. The Lie algebra $W=\Der(A)$
is called the {\it modular Witt algebra}. It has a basis
$e_i=x^{i+1}\partial$ where $i=-1,0,\dots, p-2$ and $\partial=
d/dx$. The Lie bracket in $W$ satisfies $[e_i,e_j]=(j-i)e_{i+j}$ where
$i+j$ is computed modulo $p$.

In characteristic zero, the analogous Lie algebra of derivations of
the algebra $\mathbb C [t,t^{-1}]$ is also called the Witt
algebra. Gelfand and Fuchs have shown that $\Der(\mathbb C[t,t^{-1}])$
has, up to equivalence, exactly one non-trivial one-dimensional
central extension (the Virasoro algebra) \cite{FuchsGelfand}. Block
has shown in characteristic $p>5$ that the modular Witt algebra also
has, again up to equivalence, exactly one non-trivial one-dimensional
central extension as an ordinary (i.e. non-restricted) Lie algebra
\cite{Block}.

Since $W$ is a Lie algebra of derivations of an algebra over $\K$, it
has the canonical structure of a restricted Lie algebra. It is
therefore natural to ask how many one-dimensional {\it restricted}
central extensions $W$ has. To answer this question, one needs to
compute the restricted cohomology group $H^2(W)=H^2(W,\K)$.  In this
paper we will use the (partial) complex given in \cite{EvansFuchs2} to
carry out this computation. We will give a complete computation of the
restricted cohomology groups $H^q(W)=H^q(W,\K)$ for $q=0,1$ and
$2$. We will show that if $p>3$, $W$ has $p+1$ nonequivalent
non-trivial one-dimensional restricted central extensions. Moreover,
when considered only as ordinary Lie algebra extensions, $p$ of these
extensions are trivial, and one is equivalent to the central extension
of $W$ in \cite{Block}.

The structure of the paper is as follows. In section 2, we review the
definitions of restricted Lie algebras, the (partial) cochain complex
for restricted Lie algebra cohomology for the case of trivial
coefficients, and establish the notation used throughout the paper. In
section 3, we review the correspondance between one-dimensional
central extensions and cohomology for both ordinary and restricted Lie
algebras and state our main theorem on restricted central extensions
of $W$. Section 4 gives explicit representative cocycles for the ordinary
cohomology group $H^2_{\rm cl}(W)$, and section 5 contains the
computations of the restricted cohomology groups $H^q(W)$ for $q\le
2$.  Section 5 concludes with the proof of the main theorem and
explicit descriptions of all one-dimensional restricted central
extensions of $W$.

The authors are grateful to the referee for helpful suggestions.

\section{Definitions and Notations}

Restricted Lie algebras (also called Lie $p$-algebras) were first
introduced by Jacobson \cite{Jacobson0, Jacobson}. Recall that a
modular Lie algebra $\mathfrak{g}$ is called {\it restricted} if it is
endowed with an additional unary operation $g\mapsto g^{[p]}$ that
satisfies for all $g,h\in\mathfrak{g}$ and all $\lambda\in\K$
\begin{align*}
  (\lambda g)^{[p]} & =\lambda^p g^{[p]};\\
  \ad(g^{[p]}) & = (\ad g)^p;\\
  (g+h)^{[p]} & = g^{[p]} + h^{[p]} + \sum_{i=1}^{p-1} s_i(g,h);
\end{align*}
where $is_i(g,h)$ is the coefficient of $\lambda^{i-1}$ in
$(\ad(\lambda g+h))^{p-1}(h)$.

Restricted Lie algebras naturally arise in positive characteristic as
derivation algebras of any algebra, or as the Lie algebra of algebraic
groups with the operations $[a,b]=ab-ba$ and $a^{[p]}=a^p$
\cite{Humphreys2,Jantzen}. In particular, in the Witt algebra $W$ we
have $g^{[p]}=g^p$. Whereas the operation $g\mapsto g^{[p]}$ is not
linear in general, it is completely determined by its values on a
basis. In $W$ we have $e_0^{[p]}=e_0$, $e_i^{[p]}=0$ for $i\ne 0$, and
moreover $g^{[p]}=\gamma(g)g$ where $\gamma(g)\in\K$ is a constant
(\cite{EvansFuchs1}, Theorem 1(a)). Since $W$ is simple, this
restricted structure is unique.

The classification of one-dimensional restricted central extensions of
$W$ below is carried out through the calculation of the restricted
cohomology $H^2(W)$ with coefficients in $\K$ taken as a trivial $W$
module. We limit our description of the restricted cochain spaces and
coboundary operators to the case of trivial coefficients and refer the
reader to \cite{EvansFuchs2} for the general description. For details
on ordinary Lie algebra cohomology see \cite{FuchsBook}. Let $C^q(W)$
denote restricted cochains of degree $q$ ($0\le q\le 3$) and
$C^q_{\rm cl}(W)=\Hom(\wedge^q W,\K)$ the space of ordinary Lie
algebra cochains. We use similar notation for the coboundary
operators. We will denote multiple Lie bracket products as
$[g_1,g_2,g_3,\dots, g_j]=[[\dots[[g_1,g_2],g_3],\dots,]g_j]$. In
particular, we take these products from the right so that the equality
$\ad(g^{[p]})(h)=(\ad g)^p(h)$ is written
\[[h,g^{[p]}]=[h,\underbrace{ g,\dots, g }_{p}].\]

Let us give an explicit description of the (partial) complex
\begin{diagram}
  C^0(W) & \rTo^{\delta^0} & C^1(W) & \rTo^{\delta^1} & C^2(W) &
  \rTo^{\delta^2} & C^3(W).
\end{diagram}
For $q\le 1$, we set $C^q(W)=C^q_{\rm cl}(W)=\Hom(\wedge^q W,\K)$, and
$\delta^0=\delta^0_{\rm cl}$. If $\phi\in C^2_{\rm cl}(W)$ and
$\omega:W\to\K$, then we say $\omega$ has the {\bf $*$-property} with
respect to $\phi$ if for all $g,h\in W$ and $\lambda\in\K$ we have
$\omega(\lambda g)=\lambda^p\omega (g)$ and
\begin{equation}
  \label{starprop}
  \omega (g+h)=\omega (g)+ \omega (h) + \sum_{\substack{g_i=g\ {\rm
        or}\ h\\ g_1=g,g_2=h}} \frac{1}{\#(g)}
  \phi([g_1,g_2,\dots, g_{p-1}]\wedge g_p).
\end{equation} Here $\#(g)$ is the number
of factors $g_i$ equal to $g$. We remark that $\omega$ has the
$*$-property with respect to the zero map precisely when $\omega$ is a
{\it $p$-semilinear map} on $W$, that is $\omega(g+h)=\omega(g)+\omega(h)$ and
$\omega(\lambda g)=\lambda^p\omega(g)$ for all $g,h\in W$ and all
$\lambda\in\mathbb K$. Moreover, given $\phi$, we can assign the
values of $\omega$ arbitrarily on a basis for $W$ and use
(\ref{starprop}) to define $\omega:W\to \K$ that has the $*$-property
with respect to $\phi$. Our space of 2-cochains is
\[C^2(W)= \{(\phi,\omega)\ |\ \phi:\Lambda^2 W\to \K, \omega:W\to \K\
\hbox{\rm has the}\ *\hbox{\rm -property w.r.t.}\ \phi\},\] and
\[\dim_\K C^2(W) = \frac{p(p+1)}{2}.\]
A linear map $\psi:W\to\K$ induces a map $\ind^1\psi:W\to\K$ by the
formula
\[\ind^1\psi(g)=\psi(g^{[p]}),\]
and this map has the $*$-property with respect to $\delta^1_{\rm
  cl}\psi$ (\cite{EvansFuchs2}, Lemma~4). The coboundary operator
$\delta^1:C^1(W)\to C^2(W)$ is given by
\[\delta^1\psi=(\delta^1_{\rm cl}\psi,\ind^1\psi).\]

If $\alpha:\Lambda^3 W\to\K$ is a skew-symmetric multilinear map on
$W$ and $\beta:W\times W\to \K$, we say that $\beta$ has the {\bf
  $**$-property with respect to $\alpha$,} if the following conditions
hold:
\begin{itemize}
\item[(\romannumeral 1)] $\beta (g,h)$ is linear with respect to $g$;
\item[(\romannumeral 2)] $\beta (g,\lambda h)=\lambda^p\beta(g,h)$ for
  all $\lambda\in\K$;
\item[(\romannumeral 3)] \begin{align}\label{starstarprop}
    \beta(g,h_1+h_2) &=
    \beta(g,h_1)+\beta(g,h_2)-\nonumber \\
    & \sum_{\substack{l_1,\dots,l_p=1 {\rm or} 2\\ l_1=1,
        l_2=2}}\frac{1}{\#\{i_i=1\}}\alpha (g\wedge
    [h_{l_1},\cdots,h_{l_{p-1}}]\wedge h_{l_{p}}).
  \end{align}
\end{itemize}
Again we remark that $\beta$ has the $**$-property with respect to the
zero map precisely when $\beta$ is linear in the first variable and
$p$-semilinear in the second variable. We can use
(\ref{starstarprop}) to define $\beta$ for a given $\alpha$ and values of
$\beta$ on a basis for $W$.  Our space of 3-cochains is
\[C^3(W) =\{(\alpha,\beta)\ |\ \alpha\in C^3_{\rm cl}(W),
\beta:W\times W\to\K\ \hbox {\rm has the $**$-property w.r.t.
  $\alpha$}\}\] and
\[\dim_\K C^3(W)=\frac{p(p+1)(p+2)}{6}.\]
An element $(\phi,\omega)\in C^2(W)$ induces a map
$\ind^2(\phi,\omega):W\times W\to \K$ by the formula
\begin{equation*}
  \label{3induced}
  \ind^2(\phi,\omega)(g,h)=\phi(g,h^{[p]})-\phi([g, 
  \underbrace{h,\cdots,h}_{p-1}]\wedge h),
\end{equation*}
and this map has the $**$-property with respect to $\delta^2_{\rm
  cl}\phi$ (\cite{EvansFuchs2}, Lemma~5).

The coboundary operator $\delta^2:C^2(W)\to C^3(W)$ is given by the
formula
\[\delta^2(\phi,\omega)= (\delta^2_{\rm
  cl}\phi,\ind^2(\phi,\omega)).\]

We write $e^i=e_i^*$ (dual basis vector), $e_{i,j}=e_i\wedge
e_j$, $e^{i,j}=e_{i,j}^*$ and $e_{r,s,t}=e_r\wedge e_s\wedge e_t$ where $-1\le i<j\le p-2$ and $-1\le r<s<t\le
p-2$.

The ordinary cochain spaces $C^1_{\rm cl}(W), C^2_{\rm cl}(W)$ and
$C^3_{\rm cl}(W)$ admit a natural grading:

\begin{align*}
  (C^1_{\rm cl})_k(W)  &=  \Span\{e^k\};\\
  (C^2_{\rm cl})_k(W) & =  \Span\{e^{i,j}\ |\ i+j=k\pmod p\};\\
  (C^3_{\rm cl})_k(W) &= \Span\{e^{r,s,t}\ |\ r+s+t=k\pmod p\};
\end{align*}
where $k=-1,\dots, p-2$. Moreover, we have
\begin{align*}
  \dim_\K(C^1_{\rm cl})_k(W)  &=  1; \\
  \dim_\K (C^2_{\rm cl})_k(W)  & = \frac{p-1}{2}; \\
  \dim_\K (C^3_{\rm cl})_k(W) &= \frac{(p-1)(p-2)}{6}.
\end{align*}
The coboundary operators $\delta^1_{\rm cl}$ and $\delta^2_{\rm cl}$
preserve this grading, and we denote by $(\delta^1_{\rm cl})_k$ and
$(\delta^2_{\rm cl})_k$ the restrictions of $\delta^1_{\rm cl}$ and
$\delta^2_{\rm cl}$ to $(C^1_{\rm cl})_k(W)$ and $(C^2_{\rm
  cl})_k(W)$, respectively.

\section{Restricted Central Extensions}

As stated in the Introduction, the main result of this paper is the
classification of one-dimensional restricted central extensions of
$W$. We now state the main theorem.

\begin{theorem}
  \label{maintheorem}
  If $p>3$, then $H^2(W)$ is $(p+1)$-dimensional. Moreover, there is a
  $p$-dimensional subspace of $H^2(W)$ for which each corresponding
  one-dimensional restricted central extension is trivial when
  considered as an ordinary Lie algebra extension.
\end{theorem}

Theorem~\ref{maintheorem} implies that exactly one of the $p+1$
non-trivial classes of restricted one-dimensional central extensions
remains non-trivial when considered only as a Lie algebra
extension. For $p>5$, this is the central extension described first by
Block in \cite{Block}. Our method here is different, and also gives
the result for $p=5$.

\begin{remark}
If $p=3$, the algebra $W$ is isomorphic to the Lie algebra
$\sl_2(\K)$. In this case $H^2(W)$ is 3-dimensional so that $W$ has
just three non-equivalent one-dimensional restricted central
extensions. Of course each of these extensions is trivial when
considered as an ordinary Lie algebra extension.
\end{remark}

Given a one-dimensional restricted central extension $E$ of $W$,
construct an element $(\phi,\omega)\in C^2(W)$ by choosing a
$\K$-linear splitting map $\sigma:W\to E$ and defining for all $g,h\in
W$
\begin{align}
  \label{eq:1}
  \begin{split}
    \phi(g,h) & = [\sigma(g),\sigma(h)]=\sigma([g,h]);\\
    \omega(g) & = \sigma(g)^{[p]} - \sigma(g^{[p]}).
  \end{split}
\end{align}
The element $(\phi,\omega)\in C^2(W)$ is a cocycle, and the cohomology
class of $(\phi,\omega)$ does not depend on the choice of the spitting
map, but only on the equivalence class of the extension
(\cite{EvansFuchs2}, Corollary 4).

Conversely, given a cocycle $(\phi,\omega)\in C^2(W)$, define a
restricted Lie algebra structure on $E=W\oplus \K c$ by declaring for
all $g,h\in W$ and all $\alpha,\beta\in\K$
\begin{align}
  \label{eq:2}
  \begin{split}
    [g+\alpha c, h+\beta c] & = [g,h]+\phi(g,h) c;\\
    (g+\alpha c)^{[p]} & = g^{[p]}+\omega(g) c.
  \end{split}
\end{align}
The equivalence class of the resulting one-dimensional restricted
central extension depends only on the cohomology class of
$(\phi,\omega)$ (\cite{EvansFuchs2}, Corollary 4).

The map $(\phi,\omega)\mapsto \phi$ induces a well defined map
$H^2(W)\to H^2_{\rm cl}(W)$ so that every one-dimensional restricted
central extension of $W$ is also a one-dimensional central extension
as an ordinary Lie algebra, and equivalent one-dimensional restricted
central extensions give equivalent one-dimensional ordinary central
extensions.

\begin{remark}
  It is known that modular Lie algebras do not always admit a Levi
  decomposition (a decomposition into the semi-direct product of the
  radical and a semisimple algebra) \cite{Humphreys3,Mcninch}. Our
  computations below give examples of modular Lie algebras and
  restricted Lie algebras that are not Levi decomposable. If $E$ is a
  Levi decomposable one-dimensional restricted central extension of
  $W$, then $\K$ is the radical of $E$ and there is a restricted Lie
  algebra homomorphism splitting map $\sigma:W\to E$. In this case
  $\phi$ and $\omega$ in (\ref{eq:1}) are identically zero. If $E$ is
  Levi decomposable only as an ordinary Lie algebra, then $\sigma$ is
  only a Lie algebra homomorphism. In this case $\phi=0$ but
  $\omega\ne 0$. Theorem~\ref{maintheorem} implies that there is a
  $p$-dimensional subspace of $H^2(W)$ that classifies the
  one-dimensional restricted central extensions of $W$ that are Levi
  decomposable as ordinary Lie algebras, but not as restricted Lie
  algebras. Elements in the complement to this subspace correspond to
  a one-dimensional central extension that is not Levi decomposable as
  an ordinary nor restricted Lie algebra.
\end{remark}

\section{Ordinary Lie Algebra Cohomology of $W$}

We sketch a new method for computing $H^2_{\rm cl}(W)$ by giving
explicit descriptions of the cocycles. There are more results on the
cohomology of $W$, for example \cite{Weinstein,Yakovlev}, but we will
not use them in this paper. In the case of trivial coefficients, we
have $H^0_{\rm cl}(W)=\K$ and $H^1_{\rm cl}(W)=0$ .
% for $\psi\in C^1(W)$ and $g,h\in W$
% \begin{equation}
%   \label{cl1cob}
%   \delta^1_{\rm cl}\psi(g\wedge h)=\psi([g,h]).
% \end{equation}
% As $W$ is a simple algebra, $[W,W]=W$ and hence $\delta^1_{\rm
% cl}\psi=0$ if and only if $\psi=0$ so that $H^1_{\rm cl}(W)=0$.

Since the cochain spaces are graded and the coboundary maps are graded
maps, we can compute $H^2_{\rm cl}(W)$ by computing the cohomology in
each graded component. An element $\phi\in (C^2_{\rm cl})_k(W)$ has
the form
\[\phi=\sum_{i+j=k\pmod* p} a_{i,j}e^{i,j}\] where $a_{i,j}\in\K$. The
proof of the following lemma is a routine computation.

\begin{lemma}
  \label{onecoboundary}
  For $-1\le k\le p-2$,
  \[\delta^1_{\rm cl}(e^k)=\sum_{\substack{-1\le i<j\le p-2\\
      i+j=k\pmod* p}}(j-i)e^{i,j}.\]
\end{lemma}

% \begin{proof}
%   The proof is a computation. If $-1\le i<j\le p-2$, then
%   \[ \delta^1_{\rm
%   cl}(e^k)(e_{i,j})=e^k((j-i)e_{i+j})=\left\{\begin{array}{ll}
%       j-i & \mbox{\rm if $i+j=k\pmod p$};\\
%       0 & \mbox{\rm otherwise.}
%     \end{array}\right.\]
% \end{proof}

\begin{lemma}
  \label{claim1}
  If $-1\le k\le p-2$ and $k\ne 0$, then $\dim_\K(\ker(\delta^2_{\rm
    cl})_k)=1$.
\end{lemma}

\begin{proof}
  Let $\phi=\sum a_{i,j}e^{i,j}\in (C^2_{\rm cl})_k(W)$. For $-1\le
  i<j\le p-2$, $i+j=k\pmod p$ and $i\ne 0$, we have
  \[(\delta^2_{\rm cl})_k\phi(e_{0,i,j})=ka_{i,j}-(j-i)a_{0,k}.\] If
  $k\ne 0$ and $\phi$ is a cocycle, then all coefficients $a_{i,j}$
  are determined by $a_{0,k}$ so that $\dim_\K(\ker(\delta^2_{\rm
    cl})_k)\le 1$. The rank of $(\delta^1_{\rm cl})_k$ is 1 so we
  must have $\dim_\K(\ker(\delta^2_{\rm cl})_k)=1$ as claimed.
\end{proof}

\begin{lemma}
  \label{claim2}
  $\dim_\K(\ker(\delta^2_{\rm cl})_0)=2$.
\end{lemma}

\begin{proof}
  Let
  \[\phi=a_{-1,1}e_{-1,1}+\sum_{i=2}^{(p-1)/2}
  a_{i,p-i}e^{i,p-i}\in C^2_0\] be a cochain. If $\phi$ is a cocycle,
  then for $3\le j \le (p-1)/2$ we must have
  \[\delta^2_0\phi(e_{-1,j,p-j+1})=(j+1)a_{j-1,p+j-1}+(p-j+2)a_{j,p-j}+(2j-1)a_{-1,1}=0.\]
  Shifting the index with $n=j+2$, we have for $1\le n\le (p-5)/2$,
  % \[(n+3)a_{n+1,p-n+1}+(p-n)a_{n+2,p-n-2}+(2n+3)a_{-1,1}=0.\]
  \begin{equation}\label{recursion}
    n a_{n+2,p-n-2} = (n+3)a_{n+1,p-n-1}+(2n+3)a_{-1,1}.
  \end{equation}
  This shows recursively that $a_{n+2,p-n-2}$ is determined by
  $a_{-1,1}$ and $a_{2,p-2}$, and therefore
  $\dim_\K(\ker\delta^2_0)\le 2$.  If we set $a_{-1,1}=1$ and
  $a_{2,p-2}=0$, then (\ref{recursion}) reduces to
  \[n a_{n+2,p-n-2} = (n+3)a_{n+1,p-n-1}+(2n+3).\] This recursion
  equation has the solution
  \[a_{n+2,p-n-2}=\frac{1}{3} n (n+2)(n+4),\] and therefore
  \begin{align}
    \label{virasoro}
    \phi_{1,0} & =\sum_{n=1}^{(p-1)/2} \frac{ n(n^2-4)}{3} e^{n,p-n}.
  \end{align}
  Now, for any basis vector $e_{i,j,p-i-j}\in (\wedge^3 W)_0$, we have
  by definition
  \[\delta^2_{\rm cl}\phi_{1,0}(e_{i,j,p-i-j})
  =\phi_{1,0}((j-i)e_{i+j,p-i-j} - (p-2i-j) e_{p-j,j}+(p-i-2j)
  e_{p-i,i}).\] This together with (\ref{virasoro}) shows $\phi_{1,0}$
  is a cocycle. If we set $a_{-1,1}=2$ and $a_{2,p-2}=p-4$, then
  (\ref{recursion}) reduces to the recursion equation
  \[n a_{n+2,p-n-2} = (n+3)a_{n+1,p-n-1}+(4n+6),\] which has the
  solution
  \[a_{n+2,p-n-2}=-2(n+2).\] So by Lemma~\ref{onecoboundary} we have
  \begin{align*}
    \phi_{2,p-4} & = \sum_{n=1}^{(p-1)/2} -2n e^{n,p-n}=\delta^1(e^0)
  \end{align*}
  showing $\phi_{2,p-4}$ is also a cocycle. Clearly $\phi_{2,p-4}$
 is not a multiple of $\phi_{1,0}$ because $p-4\ne 0$, which means
 $\dim_\K(\ker\delta^2_0)\ge 2$.
\end{proof}

\begin{theorem}
  $\dim_\K(H^2_{\rm cl}(W))=1$ and the cocycle $\phi_{1,0}$ in
  Lemma~\ref{claim2} generates $H^2_{\rm cl}(W)$.
\end{theorem}

\begin{proof}
  Lemma~\ref{claim1} shows that $(H^2_{\rm cl})_k(W)=0$ if $k\ne 0$
  and Lemma~\ref{claim2} shows that $(H^2_{\rm cl})_0(W)$ is one
  dimensional with $\phi_{1,0}$ spanning the non-zero cohomology
  class.
\end{proof}

\section{Restricted Lie Algebra Cohomology of $W$}
Since $\delta^0=\delta^0_{\rm cl}$ and the restricted coboundary
operator $\delta^1$ is also injective, we have $H^0(W)=\K$ and
$H^1(W)=0$.

\begin{lemma}
  \label{inducediszero}
  If $(\phi,\omega)\in C^2(W)$, then $\ind^2(\phi,\omega)=0$ and hence
  $(\phi,\omega)\in\ker \delta^2$ if and only if $\phi\in\ker
  \delta^2_{\rm cl}$.
\end{lemma}

\begin{proof} 
  It suffices to show that for all $g,h\in W$,
  \[g\wedge h^{[p]}-[g,h,\cdots,h]\wedge h=0,\] where, in this proof,
  the bracket $[g,h,\dots, h]$ always means $p-1$ factors of $h$.

  For $g,h\in W$, we have \[g\wedge h^{[p]}=g\wedge \gamma(h)h=\gamma
  (h)(g\wedge h).\] The algebra $W$ is of rank one so there is a
  nonempty Zariski open subset $U\subset W$ such if $y\in U$, $x\in W$
  and $[x,y]=0$, then $x\in\K y$. Moreover,
  \[[[g,h,\dots, h],h]=[g,h^{[p]}]=[g,\gamma (h)h]=[\gamma (h)g,h]\]
  so that
  \[[[g,h,\dots, h]-\gamma (h)g,h]=0.\] This shows that there is a
  scalar $\gamma'(g,h)\in\K$ with
  \[[g,h,\dots, h]-\gamma (h)g = \gamma'(g,h)h,\] at least for $h\in
  U$. However, the mapping $g\mapsto [g,h,\dots, h]-\gamma (h)g$ is
  algebraic, so
  \[[g,h,\dots, h]=\gamma (h)g+\gamma'(g,h) h\] for all $h\in W$ and
  hence
  \[[g,h,\dots, h]\wedge h=\gamma (h)(g\wedge
  h)+\gamma'(g,h) (h\wedge h) = \gamma (h)(g\wedge h)\] proving the
  lemma.
\end{proof}

If $\phi_{1,0}\in C^2_{\rm cl}(W)$ is the cocycle from
Lemma~\ref{claim2} and $\omega:W\to\K$ is any map with the
$*$-property with respect to $\phi_{1,0}$, Lemma~\ref{inducediszero}
implies that $(\phi_{1,0},\omega)\in C^2(W)$ is a restricted
cocycle. For $-1\le i\le p-2$, let $\omega_i:W\to\K$ be defined by
\[\omega_i(\alpha_{-1}e_{-1}+\cdots +
\alpha_{p-2}e_{p-2})=\alpha_i^p.\]

\begin{lemma}
  \label{basis}
  For $-1\le i\le p-2$, the map $\omega_i$ has the $*$-property with
  respect to $0$ and $(0,\omega_i)\in C^2(W)$ is a cocycle. Moreover,
  if $\omega:W\to\K$ is any map with the $*$-property with respect to
  $\phi_{1,0}$, the cohomology classes represented by the
  $(0,\omega_i)$ and $(\phi_{1,0},\omega)$ comprise a linearly
  independent subset in $H^2(W)$.
\end{lemma}

\begin{proof}
  If $-1\le i\le p-2$, an easy computation shows that $\omega_i$ is
  $p$-semilinear and hence has the $*$-property with respect to
  $0$. Therefore $(0,\omega_i)\in C^2(W)$ and
  $\delta^2(0,\omega_i)=(0,0)$ by Lemma~\ref{inducediszero}. If
  $\alpha_i,\beta\in\K$ and
  \[\sum
  \alpha_i(0,\omega_i)+\beta(\phi_{1,0},\omega)=\left(\beta\phi_{1,0},\sum
    \alpha_i \omega_i +\beta\omega\right )=(0,0),\] then $\beta=0$,
  and evaluating at $(0,e_j)$ shows $\alpha_j=0$. Therefore
  \[{\cal B}=\{(0,\omega_{-1}),\dots,
  (0,\omega_{p-2}),(\phi_{1,0},\omega)\}\] is a linearly independent
  set in $C^2(W)$. Moreover, if
  \[\sum
  \alpha_i(0,\omega_i)+\beta(\phi_{1,0},\omega)=\left(\beta\phi_{1,0},\sum
    \alpha_i \omega_i +\beta\omega\right )=\delta^1\psi=(\delta^1_{\rm
    cl}\psi,\ind^1\psi)\] for some $\psi\in C^1(W)$, then $\beta=0$,
  otherwise $\phi_{1,0}\in\im\delta^1_{\rm cl}$. Therefore
  $\delta^1_{\rm cl}\psi = 0$ so that $\psi=0$ and hence $\alpha_i=0$
  showing $\cal B$ is linearly independent in $H^2(W)$.
\end{proof}

\begin{proof}[Proof of Theorem~\ref{maintheorem}]
  For $i=-1, \dots, p-2$, let $\phi_i=\delta^1_{\rm cl}(e^i)$ and let
  $\phi_{p-1}=\phi_{1,0}$ so that $\{\phi_{-1},\dots,
  \phi_{p-2},\phi_{p-1}\}$ is a basis for $\ker(\delta^2_{\rm cl})$ by
  Lemmas~\ref{claim1} and \ref{claim2}. For each $i$, choose a map
  $\xi_i:W\to\K$ that has the $*$-property with respect to $\phi_i$ so
  by Lemmas~\ref{inducediszero} and \ref{basis}
  \[\{(\phi_{-1},\xi_{-1}),\dots,
  (\phi_{p-1},\xi_{p-1}),(0,\omega_1),\dots, (0,\omega_p)\}\] is a
  linearly independent subset of $\ker\delta^2$. If
  $(\phi,\omega)\in\ker\delta^2$, then $\phi\in\ker\delta^2_{\rm cl},$
  so there are scalars $\alpha_i$ ($-1\le i\le p-1$) such that
  $\phi=\sum \alpha_i\phi_i$. If $\xi=\sum\alpha_i\xi_i$, then $\sum
  \alpha_i(\phi_i,\xi_i)=(\phi,\xi)$. We have $(0,\omega-\xi)\in
  C^2(W)$ which means there are scalars $\beta_j$ ($1\le j\le p$) such
  that $\omega-\xi=\sum\beta_j\omega_j$. Therefore
  \[(\phi,\omega)=(\phi,\xi)+\sum\beta_j(0,\omega_j)=\sum\alpha_i(\phi_i,\xi_i)+\sum\beta_j(0,\omega_j).\]
  This shows
  \[\{(\phi_{-1},\xi_{-1}),\dots,
  (\phi_{p-1},\xi_{p-1}),(0,\omega_1),\dots, (0,\omega_p)\}\] is a
  basis for $\ker\delta^2$ and hence $\dim_\K\ker\delta^2=2p+1$. We
  have already seen that $\dim_\K\im\delta^1=p$ so $\dim_\K
  H^2(W)=p+1$, and from Lemma~\ref{basis} it follows that the
  cohomology classes represented by $(\phi_{p-1},\xi_{p-1}),
  (0,\omega_1),\dots, (0,\omega_p)$ form a basis for
  $H^2(W)$. Finally, the subspace spanned by the cohomology classes of
  $(0,\omega_i)$ is clearly $p$ dimensional, and the ordinary
  (non-restricted) one-dimensional central extensions of $W$
  corresponding to these cohomology classes are trivial as ordinary
  Lie algebra extensions.
\end{proof}

We conclude this section with explicit descriptions of the $p+1$
one-dimensional restricted central extensions of $W$. For $-1\le i\le
p-2$, let $E_i$ denote the one-dimensional restricted central
extension of $W$ determined by the cohomology class of the cocycle
$(0,\omega_i)$. Then $E_i=W\oplus \K c$ as a $\K$-vector space, and
using (\ref{eq:2}) we have for all $-1\le j,k\le p-2$,
\begin{align*}
  [e_j,e_k] & =(k-j) e_{j+k};\\
  [e_j, c] & = 0;\\
  e_{j}^{[p]} & = \delta_{0,j}e_0 + \delta_{i,j} c;\\
  c^{[p]} & = 0,
\end{align*}
where $\delta$ denotes Kronecker's delta-function.

Let us denote by $\phi$ the cocycle $\phi_{1,0}$ given in
(\ref{virasoro}) and define $\omega:W\to\K$ to have the $*$-property
with respect to $\phi$ using (\ref{starprop}) and declaring
$\omega(e_j)=0$ for all $-1\le j\le p-2$. Note that $\omega\ne 0$, but
$\omega(0)=0$ by (\ref{starprop}). Now, for $-1\le j,k\le p-2$,
(\ref{virasoro}) gives
\[\phi(e_{j,k})=\frac{j(j^2-4)}{3}\delta_{0,j+k}.\]
Therefore, if $E=W\oplus \K c$ denotes the one-dimensional central
extension of $W$ determined by the cocycle $(\phi,\omega)$, we have
for all $-1\le j,k\le p-2$,
\begin{align*}
  [e_j,e_k] & =(k-j) e_{j+k}+\frac{j(j^2-4)}{3}\delta_{0,j+k} c;\\
  [e_j, c] & = 0;\\
  e_{j}^{[p]} & = \delta_{0,j}e_0;\\
  c^{[p]} & = 0.
\end{align*}
The Lie bracket in the extension $E$ is similar to the bracket in the
(characteristic zero) Virasoro algebra (\cite{Schottenloher},
Def.~5.2) insofar as the coefficients in both corresponding cocycles
are given by cubic polynomials of the same form. For this reason, it
is natural to refer to the extension $E$ as the (restricted) modular
Virasoro algebra.

\bibliography{references}{} \bibliographystyle{plain}
\end{document}